\documentclass[smallcondensed]{svjour3}

\usepackage{amsmath,amssymb}
\usepackage{graphicx}
\usepackage{url}
\usepackage{xfrac}

\newcommand{\figurewidth}{\textwidth}
\newcommand{\halffigurewidth}{0.49\textwidth}
\newcommand{\real}{\mathbb{R}}
\newcommand{\cg}[1]{$\mathrm{cG}(#1)$}
\newcommand{\dg}[1]{$\mathrm{dG}(#1)$}
\newcommand{\dt}{\,\mathrm{d}t}
\newcommand{\ds}{\,\mathrm{d}s}
\newcommand{\inner}[2]{\langle #1, #2 \rangle}
\newcommand{\RR}{\bar{R}}
\newcommand{\emach}{\epsilon_{\textrm{mach}}}
\newcommand{\nmach}{n_{\textrm{mach}}}
\newcommand{\EE}{\mathbf{E}}
\newcommand{\ED}{\mathbf{E}_D}
\newcommand{\EG}{\mathbf{E}_G}
\newcommand{\EC}{\mathbf{E}_C}
\newcommand{\EQ}{\mathbf{E}_Q}
\newcommand{\kk}{\Delta{}t}

\let\oldmarginpar\marginpar
\renewcommand\marginpar[1]{\-\oldmarginpar[\raggedleft\footnotesize #1]%
{\raggedright\footnotesize #1}}

\newcommand{\assumptions}{Let $u : [0, T] \rightarrow \real^N$ be the
  exact solution of the initial value problem~\eqref{eq:u'=f}, let $z
  : [0, T] \rightarrow \real^N$ be the solution of the dual
  problem~\eqref{eq:dual}, and let $U : [0, T] \rightarrow \real^N$ be
  any piecewise smooth approximation of $u$ on a partition $0 = t_0 <
  t_1 < \cdots < t_M = T$ of $[0, T]$, that is, $U \vert_{(t_{m-1},
    t_m]} \in \mathcal{C}^{\infty}((t_{m-1}, t_m])$ for
  $m=1,2,\ldots,M$ \emph{(}$U$ is left-continuous\emph{)}.}

\begin{document}

\title{A posteriori error analysis of round-off errors in the
  numerical solution of ordinary differential equations}

\titlerunning{A posteriori error analysis of round-off errors}

\author{Benjamin Kehlet
        \and
        Anders Logg}

\institute{B. Kehlet \at
           University of Oslo and Simula Research Laboratory,
           P.O.Box 134,
           NO-1325 Lysaker,
           Norway \\
           \email{benjamik@simula.no} \\
           \and
           A. Logg \at
           Mathematical Sciences,
           Chalmers University of Technology,
           SE-41296 Gothenburg,
           Sweden \\
           \email{logg@chalmers.se}}

\date{Received: \today / Accepted: date}

\maketitle

\begin{abstract}
  We prove sharp, computable error estimates for the propagation of
  errors in the numerical solution of ordinary differential
  equations. The new estimates extend previous estimates of the
  influence of data errors and discretisation errors with a new term
  accounting for the propagation of numerical round-off errors,
  showing that the accumulated round-off error is inversely
  proportional to the square root of the step size. As a consequence,
  the numeric precision eventually sets the limit for the pointwise
  computability of accurate solutions of any ODE. The theoretical
  results are supported by numerically computed solutions and error
  estimates for the Lorenz system and the van der Pol oscillator.
\end{abstract}


\section{Introduction}
We consider the numerical solution of general initial value problems
for systems of ordinary differential equations (ODE),
\begin{equation} \label{eq:u'=f}
  \begin{split}
    \dot{u}(t) &= f(u(t),t), \quad t \in (0, T],
      \\ u(0) &= u_0,
  \end{split}
\end{equation}
where the right-hand side $f : \real^N \times [0, T] \rightarrow
\real^N$ is assumed to be Lipschitz continuous in~$u$ and continuous
in~$t$. Our objective is to analyse the error in an approximate
solution~$U : [0, T] \rightarrow \real^N$ computed by a single-step
numerical method, such as an explicit or implicit Runge--Kutta method.
For the numerical results presented at the end of this work, we have
used a time-stepping method formulated as a Galerkin finite element
method --- which, for any particular choice of finite element basis
and quadrature, will correspond to a particular implicit Runge--Kutta
method --- but stress that the theoretical results are \emph{not}
particular to time-stepping based on Galerkin formulations but generic
to all single-step methods.

The propagation of local errors and accumulation of global errors in
the numerical solution of ODE have been studied extensively in the literature, see
e.g.
\cite{eriksson_introduction_1995,estep_global_1994,estep_posteriori_1995,estep_pointwise_1998,cao2004posteriori}.
These estimates are based on the formulation of an auxiliary dual
problem: the linearised adjoint problem. From the solution of the dual
problem, the accumulation rate of local errors may be computed, either
as global stability factors or as local stability weights. These
factors or weights, together with a measure of the local error,
typically the residual $R(t) = \dot{U} - f(U(t), t)$ lead to a
computable estimate of the global error.

Standard estimates may include various sources contributing to the
global error, such as discretisation errors, accounting for the use of
finite time steps, quadrature errors, accounting for the approximation
of the right-hand side $f$ by a particular quadrature rule, and data
errors, accounting for the approximation of the initial value $u_0$.
In this work, we extend these estimates by adding a new term
accounting for the use of finite numeric precision in the
computation of the numerical solution. This error is normally
neglected, since it is typically much smaller than the contribution
from the data or discretisation error. However, when the
system~\eqref{eq:u'=f} is very sensitive to perturbations, when the
time interval $[0, T]$ is very long, or when a solution is sought with
very high accuracy, the effect of numerical round-off errors as a
result of finite numeric precision can and will be the dominating
error source, which ultimately limits the computability of a given
problem.

\section{Main result}

We prove below that the error $\EE$ in a computed numerical
solution~$U$ approximating the exact solution~$u$ of the
ODE~\eqref{eq:u'=f} is the sum of three contributions:
\begin{equation*}
  \EE = \ED + \EG + \EC,
\end{equation*}
where $\ED$ is the data error, which is nonzero if $U(0) \neq u(0)$;
$\EG$ is the discretisation error, which is nonzero as a result of a
finite time step; and $\EC$ is the computational error, which is
nonzero as a result of finite numerical precision. Furthermore, we
bound each of the three contributions as the product of a stability
factor and a residual that measures the size of local contributions to
the error. The size of the residuals may be estimated in terms of the
size of the time step. We find that
\begin{equation*}
  \EE \sim S_D(T) \|U(0) - u(0)\| + S_G(T) \kk^{r} + S_C(T) \kk^{-1/2},
\end{equation*}
where $\kk$ is the size of the time step, $r$ the order of convergence
of the numerical method, and $S_D(T)$, $S_G(T)$, $S_C(T)$ are
stability factors which can be computed a~posteriori.

\section{Error analysis}
\label{sec:erroranalysis}

The error analysis is based on the solution of an auxiliary \emph{dual
  problem} and follows the now well established techniques developed
in~\cite{estep_global_1994}, \cite{estep_posteriori_1995},
\cite{eriksson_introduction_1995} and \cite{becker_optimal_2001}, with
extensions to account for the accumulation of round-off errors.

The dual (linearised adjoint) problem takes the form of an initial
  value problem for a system of linear ordinary differential
  equations:
\begin{equation} \label{eq:dual}
  \begin{split}
    - \dot{z}(t) &= \bar{A}^{\top}(t) z(t), \quad t \in [0, T), \\
    z(T) &= z_T.
  \end{split}
\end{equation}
Here, $\bar{A}$ denotes the Jacobian matrix of the right-hand side~$f$
averaged over the approximate solution~$U$ and the exact solution~$u$:
\begin{equation} \label{eq:average}
  \bar{A}(t) = \int_0^1 \frac{\partial f}{\partial u}(sU(t) + (1-s)u(t), t) \ds.
\end{equation}
By the chain rule, it follows that
\begin{equation*}
  \begin{split}
    \bar{A}(t) (U(t) - u(t))
    &= \int_0^1 \frac{\partial f}{\partial u}(sU(t) + (1 - s)u(t), t)(U(t) - u(t)) \ds \\
    &= \int_0^1 \frac{\partial}{\partial s} f(sU(t) + (1 - s)u(t), t) \ds
    = f(U(t), t) - f(u(t), t).
  \end{split}
\end{equation*}

Based on the formulation of the dual problem we may now derive a
(standard) error representation
(Theorem~\ref{th:errorrepresentation}). It represents the error in an
approximate solution~$U$ (computed by any numerical method) in terms
of the residual~$R$ of the computed solution and the solution~$z$ of
the dual problem~\eqref{eq:dual}. The only assumption we make on the
numerical solution~$U$ is that it is piecewise smooth on a partition
of the interval~$[0, T]$ (or that it may be extended to such a
function). At points where $U$ is smooth, the residual is defined by
\begin{equation*}
  R(t) = \dot{U}(t) - f(U(t), t).
\end{equation*}

\begin{theorem}[Error representation] \label{th:errorrepresentation}
  \assumptions{}
  Then, the error $U(T) - u(T)$ may be represented by
  \begin{equation*}
    \inner{z_T}{U(T) - u(T)} =
    \inner{z(0)}{U(0) - u(0)} +
    \sum_{m=1}^M \inner{z(t_{m-1})}{[U]_{m-1}} +
    \int_0^T \inner{z}{R} \dt,
  \end{equation*}
  where $R(t) = \dot{U}(t) - f(U(t), t)$ is the residual of the
  approximate solution~$U$ and $[U]_{m-1} = U(t_{m-1}^+) - U(t_{m-1})
  = \lim_{t\rightarrow t_{m-1}^+} U(t) - U(t_{m-1})$.
\end{theorem}
\begin{proof}
  By the definition of the dual problem, we find that
  \begin{equation*}
    \inner{z_T}{e(T)}
    = \inner{z_T}{e(T)} - \int_0^T \inner{\dot{z} + \bar{A}^{\top}z}{e} \dt
    = \inner{z_T}{e(T)} - \sum_{m=1}^M\int_{t_{m-1}}^{t_m} \inner{\dot{z} + \bar{A}^{\top}z}{e} \dt,
  \end{equation*}
  where $e = U - u$. Noting that $\inner{\bar{A}^{\top}z}{e} =
  \inner{z}{\bar{A}e}$ and integrating by parts, we obtain
  \begin{equation*}
    \inner{z_T}{e(T)} =
    \inner{z(0)}{e(0)} +
    \sum_{m=1}^M \left[ \inner{z(t_{m-1})}{[U]_{m-1}} +
    \int_{t_{m-1}}^{t_m} \inner{z}{\dot{e} - \bar{A}e} \dt \right],
  \end{equation*}
  where $[U]_{m-1} = U(t_{m-1}^+) - U(t_{m-1}^-) = U(t_{m-1}^+) -
  U(t_{m-1})$ denotes the jump of~$U$ at $t = t_{m-1}$. By the
  construction of $\bar{A}$, it follows that $\bar{A}e = f(U, \cdot) -
  f(u, \cdot)$. Hence, $\dot{e} - \bar{A}e = \dot{U} - f(U, \cdot) -
  \dot{u} + f(u, \cdot) = \dot{U} - f(U, \cdot) = R$, which completes the proof.
\end{proof}

\begin{remark}
  Theorem~\ref{th:errorrepresentation} holds for any piecewise smooth
  function $U : [0, T] \rightarrow \real^N$, in particular for any
  piecewise smooth extension of any approximate numerical solution
  obtained by any numerical method for~\eqref{eq:u'=f}.
\end{remark}

We next investigate the contribution to the error in the computed
numerical solution~$U$ from errors in initial data, numerical
discretisation, and computation (round-off errors),
$\EE = \ED + \EG + \EC$, and derive sharp bounds for each term.

To estimate the computational error, we introduce the \emph{discrete
  residual}~$\RR$ defined as follows. For any $p \geq 0$, let
$\{\lambda_k\}_{k=0}^p$ be the Lagrange nodal basis for
$\mathcal{P}^p([0, 1])$, the space of polynomials of degree $\leq p$
on $[0,1]$, on a partition $0 \leq \tau_0 < \tau_1 < \cdots < \tau_p
\leq 1$ of~$[0,1]$, that is, $\mathrm{span} \{\lambda_k\}_{k=0}^p =
\mathcal{P}^p([0, 1])$ and $\lambda_i(\tau_j) = \delta_{ij}$. Then,
the discrete residual $\RR_k$ is defined on each interval $(t_{m-1},
t_m]$ by
\begin{equation} \label{eq:discreteresidual}
  \RR_k^m =
  \lambda_k(0) [U]_{m-1} +
  \int_{t_{m-1}}^{t_m} \lambda_k((t - t_{m-1})/\kk_m) R(t) \dt,
  \quad k=0,1,\ldots,p.
\end{equation}
We also define the corresponding interpolation operator $\pi$ onto
the space of piecewise polynomial functions on the partition $0 = t_0
< t_1 < \cdots < t_M = T$ by
\begin{equation*}
  (\pi v)(t) = \sum_{k=0}^p v(t_{m-1} + \tau_k\kk_m) \, \lambda_k((t - t_{m-1}) / \kk_m),
  \quad t \in (t_{m-1}, t_m].
\end{equation*}
We may now prove the following a~posteriori error estimate.
\begin{theorem}[Error estimate] \label{th:errorestimate}
  \assumptions{}
  Then, for any $p \geq 0$, the following error estimate holds:
  \begin{equation} \label{eq:errorestimate}
    \inner{z_T}{U(T) - u(T)} = \ED + \EG + \EC,
  \end{equation}
  where
  \begin{equation*}
    \begin{split}
      |\ED| &\leq S_D \, \|U(0) - u(0)\|, \\
      |\EG| &\leq S_G \, C_p \max_{[0,T]} \left\{\kk^{p+1} (\|[U]\|/\kk + \|R\|)\right\}, \\
      |\EC| &\leq S_C \, C_p' \max_{0\leq k \leq p} \max_{[0,T]} \|\kk^{-1} \RR_k\|.
    \end{split}
  \end{equation*}
  Here, $C_p$ and $C_p'$ are constants depending only on~$p$.
  The stability factors $S_D$, $S_G$, and $S_C$ are defined by
  \begin{equation*}
    \begin{split}
      S_D &= \|z(0)\|, \\
      S_G &= \int_0^T \|z^{(p+1)}\| \dt, \\
      S_C &= \int_0^T \|\pi z\| \dt.
    \end{split}
  \end{equation*}
\end{theorem}
\begin{proof}
  Starting from the error representation of
  Theorem~\ref{th:errorrepresentation}, we add and subtract the
  degree~$p$ left-continuous piecewise polynomial interpolant~$\pi z$
  defined above to obtain
  \begin{equation*}
    \begin{split}
      \inner{z_T}{e(T)}
      &=
      \inner{z(0)}{e(0)} \\
      &\quad+
      \sum_{m=1}^M
      \left[
      \inner{z(t_{m-1}) - \pi z(t_{m-1}^+)}{[U]_{m-1}} +
      \int_{t_{m-1}}^{t_m} \inner{z - \pi z}{R} \dt
      \right] \\
      &\quad+
      \sum_{m=1}^M
      \left[
      \inner{\pi z(t_{m-1}^+)}{[U]_{m-1}} +
      \int_{t_{m-1}}^{t_m} \inner{\pi z}{R} \dt
      \right] \\
      &\equiv
      \ED + \EG + \EC.
    \end{split}
  \end{equation*}
  We first note that the data error~$\ED$ is bounded by $\|z(0)\|
  \, \|e(0)\| \equiv S_D \, \|e(0)\|$. By an interpolation estimate,
  we may estimate the discretisation error~$\EG$ by
  \begin{equation*}
    \begin{split}
      \EG
      &\leq
      \sum_{m=1}^M
      \left[
      \|z(t_{m-1}) - \pi z(t_{m-1}^+)\| \, \|[U]_{m-1}\| +
      \int_{t_{m-1}}^{t_m} \|z - \pi z\| \, \|R\| \dt
      \right] \\
      &\leq
      C_p \max_{[0,T]} \left\{\kk^{p+1} (\|[U]\|/\kk + \|R\|) \right\}
      \sum_{m=1}^M \int_{t_{m-1}}^{t_m} \|z^{(p + 1)}\| \dt,
    \end{split}
  \end{equation*}
  where $\sum_{m=1}^M \int_{t_{m-1}}^{t_m} \|z^{(p+1)}\| \dt =
  \int_0^T \|z^{(p+1)}\| \dt \equiv S_G$ and $C_p$ is an interpolation
  constant. Finally, to estimate the computational error, we expand
  $\pi z$ in the nodal basis to obtain
  \begin{equation*}
    \begin{split}
      \EC
      &=
      \sum_{m=1}^M
      \sum_{k=0}^p
      \left\langle
      z(t_{m-1} + \tau_k\kk_m),
      \lambda_k(0) {[U]_{m-1}} +
      \int_{t_{m-1}}^{t_m} \lambda_k((t - t_{m-1})/\kk_m) R(t) \dt
      \right\rangle \\
      &=
      \sum_{m=1}^M
      \sum_{k=0}^p
      \inner{z(t_{m-1} + \tau_k\kk_m)}{\RR_k^m}
      =
      \sum_{m=1}^M
      \kk_m
      \sum_{k=0}^p
      \inner{z(t_{m-1} + \tau_k\kk_m)}{\kk_m^{-1} \RR_k^m} \\
      &\leq
      \sum_{m=1}^M
      \kk_m
      \sum_{k=0}^p
      \|z(t_{m-1} + \tau_k\kk_m)\| \, \|\kk_m^{-1} \RR_k^m\| \\
      &\leq
      \max_{0\leq k \leq p} \max_{[0,T]} \|\kk^{-1} \RR_k\|
      \sum_{m=1}^M \kk_m \sum_{k=0}^p  \|z(t_{m-1} + \tau_k\kk_m)\| \\
      &\leq
      C_p' \max_{0 \leq k \leq p} \max_{[0,T]} \|\kk^{-1} \RR_k\|
      \sum_{m=1}^M \int_{t_{m-1}}^{t_m} \|\pi z\| \dt,
    \end{split}
  \end{equation*}
  where $\sum_{m=1}^M \int_{t_{m-1}}^{t_m} \|\pi z\| \dt = \int_0^T
  \|\pi z\| \dt \equiv S_C$ and $C_p'$ is a constant depending only
  on $p$. This completes the proof.
\end{proof}

\begin{remark}
  Theorem~\ref{th:errorestimate} estimates the size of
  $\inner{z_T}{U(T) - u(T)}$ for any given vector~$z_T$. We may thus
  estimate any bounded linear functional of the error at the final
  time by choosing $z_T$ as the corresponding Riesz representer. In
  particular, we may estimate the error in any component $u_i(T)$ of
  the solution by setting $z_T$ to the $i$th unit vector for $i =
  1,2,\ldots,N$.
\end{remark}

Theorem~\ref{th:errorestimate} extends standard a~posteriori error
estimates for systems of ordinary differential equations in two
ways. First, it does not make any assumption on the underlying
numerical method. Second, it includes the effect of numerical
round-off errors. A similar estimate can be found
in~\cite{logg_multi-adaptive_2003-1} but only for the simplest case of
the piecewise linear \cg{1}~method (Crank--Nicolson).

We now investigate the propagation of numerical round-off errors in
more detail. As in the proof of Theorem~\ref{th:errorestimate}, $\EC$
denotes the computational error defined by
\begin{equation}
  \EC =
  \sum_{m=1}^M
  \left[
  \inner{\pi z(t_{m-1}^+)}{[U]_{m-1}} +
  \int_{t_{m-1}}^{t_m} \inner{\pi z}{R} \dt \right].
\end{equation}
Theorem~\ref{th:errorestimate} bounds the computational error in terms
of the discrete residual defined in~\eqref{eq:discreteresidual}. The
discrete residual tests the continuous residual $R = \dot{U} - f$
of~\eqref{eq:u'=f} against polynomials of degree $p$. In particular,
it tests how well the numerical method satisfies the relation
\begin{equation} \label{eq:timestep}
  U(t_m) = U(t_{m-1}) + \int_{t_{m-1}}^{t_m} f(U, \cdot) \dt.
\end{equation}
With a machine precision of size $\emach$, our best hope is that the
numerical method satisfies~\eqref{eq:timestep} to within a tolerance
of size $\emach$ for each component of the vector~$U$. It
follows by the Cauchy--Schwarz inequality that
\begin{equation*}
  \max_{k,m} \|\RR_k^m\| \leq \emach \sqrt{N}.
\end{equation*}
We thus have the following corollary.

\begin{corollary} \label{cor:ec}
  The computational error $\EC$ of Theorem~\ref{th:errorestimate} is
  bounded by
  \begin{equation*}
    |\EC| \leq S_C \, C_p' \frac{\emach\sqrt{N}}{\min_{[0,T]} \kk}.
  \end{equation*}
\end{corollary}
This indicates that the computational error scales like $\kk^{-1}$; the
smaller the time step, the larger the computational error. At first,
this seems non-intuitive, but it is a simple consequence of the fact
that a smaller time step leads to a larger number of time steps and
thus a larger number of round-off errors.

The estimate of Corollary~\ref{cor:ec} is overly pessimistic. It is
based on the assumption that round-off errors accumulate without
cancellation. In practice, the round-off error is sometimes positive
and sometimes negative. As a simple model, we make the assumption that
the round-off error is a random variable which takes the value
$+\emach$ or $-\emach$ with equal probabilities,
\begin{equation} \label{eq:random}
  (\RR_k^m)_i =
  \left\{
  \begin{array}{rcl}
    +\emach, &\quad& p = 0.5, \\
    -\emach, &\quad& p = 0.5, \\
  \end{array}
  \right.
\end{equation}
for all $m, k, i$. In reality, round-off errors are not uncorrelated
random variables, but the simple model~(\ref{eq:random}) may still
give useful results. For a discussion on the applicability of random
models to the propagation of round-off errors,
see~\cite{higham2002accuracy} (Section~2.8) and~\cite{Hairer2008}.

Under the assumption~\eqref{eq:random}, we find that the expected size
of the computational error scales like $\kk^{-1/2}$. As we shall see in
the next section, this is also confirmed by numerical experiments.  A
similar result is obtained in a series of papers by
Li~et.~al~\cite{li2000,li2001}. In \cite{li2000}, it is first noted
that there exists an \emph{optimal time step}; that is, a time step
for which discretisation errors and round-off errors balance. In
\cite{li2001}, it is then found that the round-off error is inversely
proportional to the square root of the time step. These results are
confirmed by the following theorem.

\begin{theorem} \label{th:random}
  Assume that the round-off error is a random variable of size
  $\pm\emach$ with equal probabilities. Then the root-mean squared
  expected computational error $\EC$ of Theorem~\ref{th:errorestimate}
  is bounded by
  \begin{equation*}
    (E[\EC^2])^{1/2}
    \leq
    S_{{C_2}} \, \sqrt{C_p'} \frac{\emach}{\min_{[0,T]}\sqrt{\kk}},
  \end{equation*}
  where
  $S_{C_2} = \left(\int_0^T \|\pi z\|^2 \dt\right)^{1/2}$
  and $C_p'$ is a constant depending only on~$p$.
\end{theorem}
\begin{proof}
  As in the proof of Theorem~\ref{th:errorestimate}, we obtain
  \begin{equation*}
    \EC
    =
    \sum_{m=1}^M \sum_{k=0}^p
    \inner{z(t_{m-1} + \tau_k\kk_m)}{\RR_k^m}
    =
    \sum_{m=1}^M \sum_{k=0}^p \sum_{i=1}^N
    z_i(t_{m-1} + \tau_k\kk_m) (\RR_k^m)_i,
  \end{equation*}
  where by assumption $(\RR_k^m)_i = \emach x_{mki}$ and $x_{mki} =
  \pm 1$ with probability $0.5$ and $0.5$, respectively. It follows
  that
  \begin{equation*}
    \begin{split}
      \EC^2
      &=
      \sum_{m,n=1}^M \sum_{k,l=0}^p \sum_{i,j=1}^N
      z_i(t_{m-1} + \tau_k \kk_m) z_j(t_{n-1} + \tau_l \kk_n) \,
      \emach^2 x_{mki} x_{nlj} \\
      &=
      \sum_{(m,k,i) = (n,l,j)}
      z_i^2(t_{m-1} + \tau_k \kk_m) \, \emach^2 x_{mki}^2 \\
      &+
      \sum_{(m,k,i) \neq (n,l,j)}
      z_i(t_{m-1} + \tau_k \kk_m) z_j(t_{n-1} + \tau_l \kk_n) \,
      \emach^2 x_{mki} x_{nlj}.
    \end{split}
  \end{equation*}
  We now note that $x_{mki}^2 = 1$. Furthermore, $y_{ijklmn} = x_{mki}
  x_{nlj}$ is a random variable which takes the values $+1$ and $-1$
  with equal probabilities. We thus find that
  \begin{equation*}
    \begin{split}
      E[\EC^2]
      &=
      \emach^2
      \sum_{m=1}^M \sum_{k=0}^p \sum_{i=1}^N
      z_i^2(t_{m-1} + \tau_k \kk_m) + 0 \\
      &=
      \emach^2 \sum_{m=1}^M \sum_{k=0}^p \|z(t_{m-1} + \tau_k \kk_m)\|^2 \\
      &\leq
      \frac{\emach^2}{\min_{[0,T]}\kk} \sum_{m=1}^M \kk_m \sum_{k=0}^p \|z(t_{m-1} + \tau_k \kk_m)\|^2
      \leq
      S_{{C_2}}^2 \, C_p' \frac{\emach^2}{\min_{[0,T]}\kk},
    \end{split}
  \end{equation*}
  where $S_{C_2} = \left(\int_0^T \|\pi z\|^2
  \dt\right)^{1/2}$. This completes the proof.
\end{proof}

\begin{remark}
  By additional assumptions on the smoothness of the dual
  solution~$z$, one may relate the error to the expected value of
  the distance from the starting point for a random walk which is
  $\sqrt{n}$ for $n$ steps (for $n$ large), and prove a similar
  estimate for the expected absolute value of the computational error,
  $E[|\EC|] \sim S_C \emach / \sqrt{\kk}$, where $S_C = \int_0^T
  \|\pi z\| \dt$.
\end{remark}

\begin{remark}
  By Cauchy--Schwarz, the stability factor $S_C$ of
  Theorem~\ref{th:errorestimate} is bounded by $\sqrt{T} S_{C_2}$.
\end{remark}

\begin{remark}
  In \cite{Hairer2008}, the effect of numerical round-off error
  accumulation and its relation to Brownian motion (Brouwer's law) are
  discussed in the context of symplectic methods for Hamiltonian
  systems. It should be noted that although the assumptions of
  Theorem~\ref{th:random} are similar to those in \cite{Hairer2008},
  namely that the process of error accumulation for round-off errors
  is random rather than systematic, the point under discussion in the
  present work is different: the effect of time step size rather than
  the effect of the interval length.
\end{remark}

We conclude this section by discussing how the above error estimates
apply to the particular methods used in this work. The estimate of
Theorem~\ref{th:errorestimate} is valid for any numerical method but
is of particular interest as an a~posteriori error estimate for the
finite element methods \cg{q} and \dg{q}
(see~\cite{hulme_discrete_1972,hulme_one-step_1972,johnson_error_1988,delfour_discontinuous_1981}).

The continuous and discontinuous Galerkin methods \cg{q} and \dg{q}
are formulated by requiring that the residual $R = \dot{U} -
f(U,\cdot)$ be orthogonal to a suitable space of test functions. By
making a piecewise polynomial ansatz, the solution may be computed on
a sequence of intervals partitioning the computational domain $[0, T]$
by solving a system of equations for the degrees of freedom on each
consecutive interval. For a particular choice of numerical quadrature
and degree~$q$, the \cg{q} and \dg{q} methods both reduce to standard
implicit Runge--Kutta methods.

In the case of the \cg{q} method, the numerical solution~$U$ is a
continuous piecewise polynomial of degree~$q$ that on each
interval~$(t_{n-1}, t_n]$ satisfies
\begin{equation} \label{eq:cgq}
  \int_{t_{n-1}}^{t_n} v \, R \dt = 0
\end{equation}
for all $v \in \mathcal{P}^{q-1}([t_{n-1}, t_n])$. It follows that the
discrete residual~\eqref{eq:discreteresidual} is zero if $p \leq q -
1$. However, this is only true in exact arithmetic. In practice, the
discrete residual is nonzero and measures how well we solve the \cg{q}
equations~\eqref{eq:cgq}, including round-off errors and errors from
numerical quadrature.\footnote{To account for additional quadrature
  errors present if the integral of~\eqref{eq:cgq} is approximated by
  quadrature, one may add and subtract an interpolant $\pi f$ of the
  right-hand side $f$ in the proof of Theorem~\ref{th:errorestimate}
  to obtain an additional term $\EQ = S_Q \max_{[0, T]} \|\pi f - f\|$
  where $S_Q = \int_0^T \|z\| \dt \approx S_C$.} For the \cg{q}
method, we further expect the residual to converge as $\kk^q$. Thus,
choosing $p = q - 1$ in Theorem~\ref{th:errorestimate}, one may expect
the error for the \cg{q} method to scale as
\begin{equation} \label{eq:convergence}
  E = \ED + \EG + \EC
  \sim S(T) \left( \emach + \kk^{2q} + \kk^{-1/2} \emach \right).
\end{equation}
Here, $S(T)$ denotes a generic stability factor. As in
Theorem~\ref{th:errorestimate}, each term contributing to the total
error is in reality multiplied by a particular stability factor. In
practice, however, the growth rates of the different stability factors
are similar and related by a constant factor.

\section{Numerical results}
\label{sec:numericalresults}

In this section, we present numerical results in support of
Theorem~\ref{th:errorestimate} and Theorem~\ref{th:random}.  The
examples are the well-known Lorenz system and Van der Pol
oscillator. Both examples illustrate the competing convergence rates
for discretisation errors, decreasing rapidly for smaller time steps,
and computational errors (round-off error), increasing for
smaller time steps.

The numerical results were obtained using the authors' software package
Tanganyika \cite{tanganyika-web} which implements the methods
described in \cite{logg_multi-adaptive_2003-1} using high precision
numerics provided by GMP~\cite{Granlund15}.  A complete code for
reproducing all results in this paper is available at
\cite{kehlet_2015_16671}. For details on the implementation,
see~\cite{kehlet_analysis_2010}.

\subsection{The Lorenz system}

We first consider the well-known Lorenz
system~\cite{lorenz_deterministic_1963}, a simple system of three
ordinary differential equations exhibiting rapid amplification of
numerical errors:
\begin{equation} \label{eq:lorenz}
  \left\{
  \begin{aligned}
    \dot{x} &= \sigma (y-x), \\
    \dot{y} &= rx - y -xz,   \\
    \dot{z} &= xy - bz,      \\
    \end{aligned}
  \right.
\end{equation}
where $\sigma = 10$, $b = 8/3$, and $r = 28$. We take $u(0) = (1, 0, 0)$.

The Lorenz system is deterministically chaotic. In the context of
a~posteriori error analysis of numerical methods for the solution of
ODE initial value problems, as in the present work, this means that
solutions may, in principle, be computed over arbitrarily long time
intervals, but to a rapidly increasing cost as function of the final
time $T$.

\subsubsection{Computability and growth of stability factors}

In~\cite{estep_pointwise_1998}, computability was demonstrated and
quantified for the Lorenz on time intervals of moderate length ($T =
30$) on a standard desktop computer. This result was further extended
to time $T = 48$ in~\cite{logg_multi-adaptive_2003}, using high order
($\|e(T)\| \sim \kk^{30}$) finite element methods. Solutions over
longer time intervals have been computed based on shadowing (the
existence of a nearby exact solution),
see~\cite{coomes_rigorous_1995}, but for unknown initial data.
Related work on high-precision numerical methods applied to the Lorenz
system include~\cite{viswanath2004fractal}
and~\cite{jorba_software_2005}.

In~\cite{KehletLogg2014}, the authors study the computability of the
Lorenz system in detail on the time interval $[0,
1000]$. Computability is here defined as the maximal final time $T =
T(\emach)$ such that a solution may be computed with a given machine
precision $\emach$. The computability may be estimated by examining
the growth rate of the stability factors appearing in the error
estimate of Theorem~\ref{th:errorestimate}. By numerical solution of
the dual problem, it was found in~\cite{KehletLogg2014} that the
stability factors grow exponentially as $S(T) \sim 10^{0.388 T} \sim
10^{0.4 T}$; see Figure~\ref{fig:stability}.
\begin{figure}
  \begin{center}
    \includegraphics[width=\halffigurewidth]{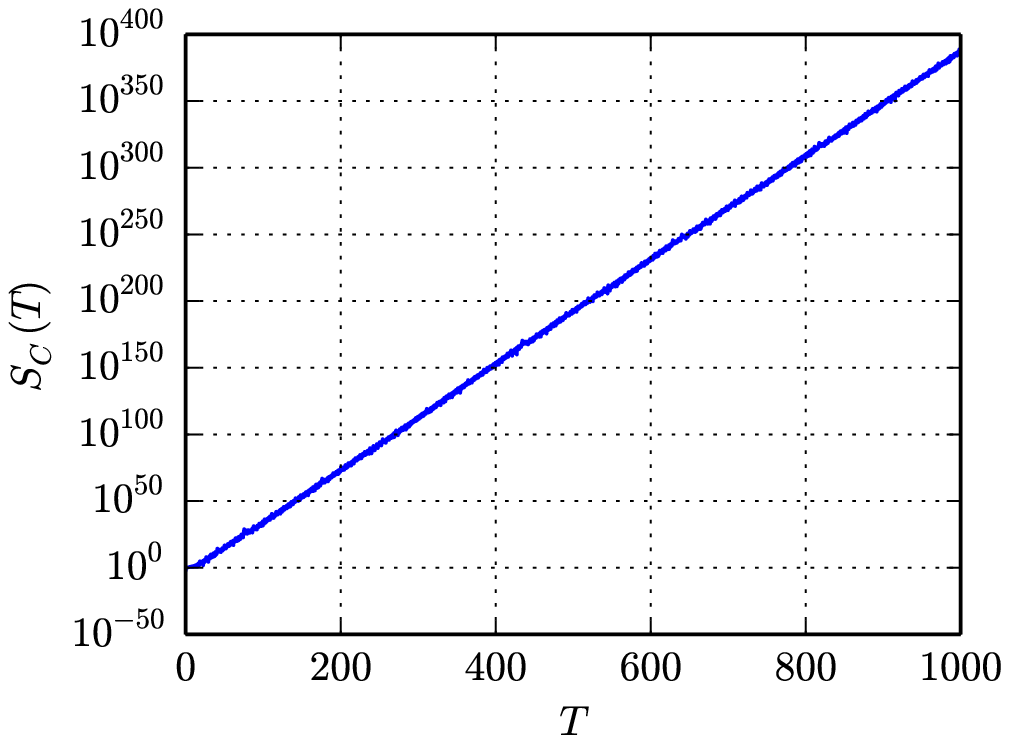}
    \includegraphics[width=\halffigurewidth]{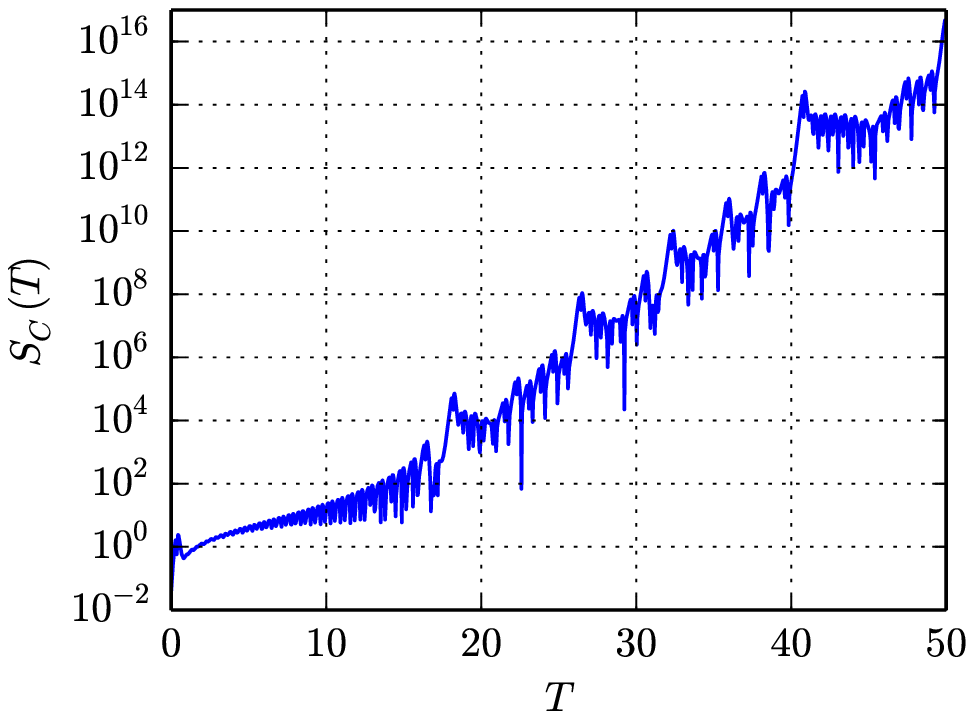}
    \caption{Growth of the stability factor $S_C$ (left) for the
      Lorenz system on the time interval~$[0, 1000]$ and a detailed
      plot on the time interval $[0, 50]$ (right).}
    \label{fig:stability}
  \end{center}
\end{figure}
By examining in detail the terms contributing to the error
estimate~\eqref{eq:errorestimate}, one finds that an optimal step
size is given by $\kk \sim \emach^{\frac{1}{2q + \sfrac{1}{2}}}$ and
that the computability of the Lorenz system is given by
\begin{equation}
  T(\emach) \sim 2.5 \nmach,
\end{equation}
where $\nmach = -\log_{10} \emach$ is the number of significant
digits. Based on this estimate, one may conclude that with 16-digit
precision, the Lorenz system is computable on $[0, 40]$, while
using 400 digits, the Lorenz system is computable on $[0, 1000]$.

In Figure~\ref{fig:solution}, we plot the solution of the Lorenz
system on the interval~$[0, 1000]$. The solution was computed with
\cg{100}, which is a method of order $2q = 200$, a time step of size
$\kk = 0.0037$, 420-digit precision arithmetic\footnote{The requested
  precision from GMP was 420 digits. The actual precision is somewhat
  higher depending on the number of significant bits chosen by GMP.},
and a tolerance for the discrete residual of size $\emach \approx 2.26
\cdot 10^{-424}$. The very rapid (exponential) accumulation of
numerical errors makes the Lorenz ``fingerprint'' displayed in
Figure~\ref{fig:solution} useful as a reference for verification of
solutions of the Lorenz system. If a solution is only slightly wrong,
the error is quickly magnified so that the error becomes visible by a
direct inspection of a plot of the solution.

\begin{figure}
  \begin{center}
    \includegraphics[width=\figurewidth]{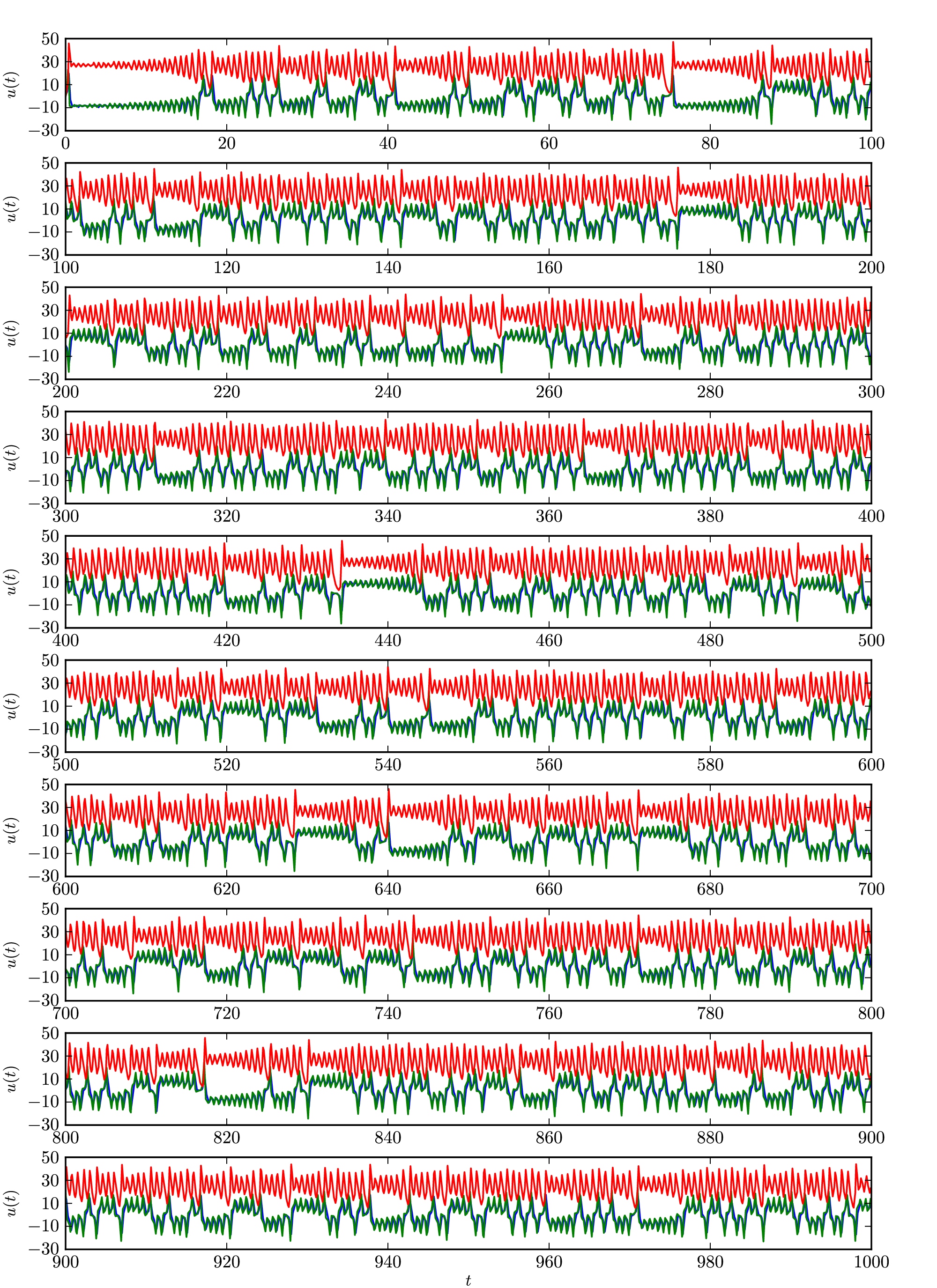}
    \caption{Accurate reference solution for the three components of the Lorenz system on the
      interval~$[0, 1000]$ with the $x$ and $y$ components plotted in
      blue and green respectively (and almost overlaid) and the $z$
      component in red.}
    \label{fig:solution}
  \end{center}
\end{figure}

\subsubsection{Order of convergence and optimal step size}

We next investigate how the accumulated error at final time depends on
the size of the time step $\kk$. According to~\eqref{eq:convergence}, we
expect the error to scale like $\kk^{2q} + \kk^{-\sfrac{1}{2}}
\emach$. Thus, for a gradually decreasing step size, we expect the
error to decrease at a rate of $\kk^{2q}$. However, as the time step
becomes smaller the second term $\kk^{-\sfrac{1}{2}}$ will grow and, for
small enough $\kk$, be the dominating contribution to the error.  This
picture is confirmed by the results presented in
Figure~\ref{fig:errors} for two numerical methods, the $2$nd order
\cg{1} and the $10$th order \cg{5} method. Of particular interest in
this figure is the very short range in which the $10$th order
convergence of the \cg{5} method is recovered; with only 16 digits of
precision, the dominating contribution to the total error is the
accumulated round-off error. We also note that for both methods, one
may find an optimal size of the time step $\kk$ for which both
contributions to the total error are balanced.

\begin{figure}
  \begin{center}
    \includegraphics[width=\halffigurewidth]{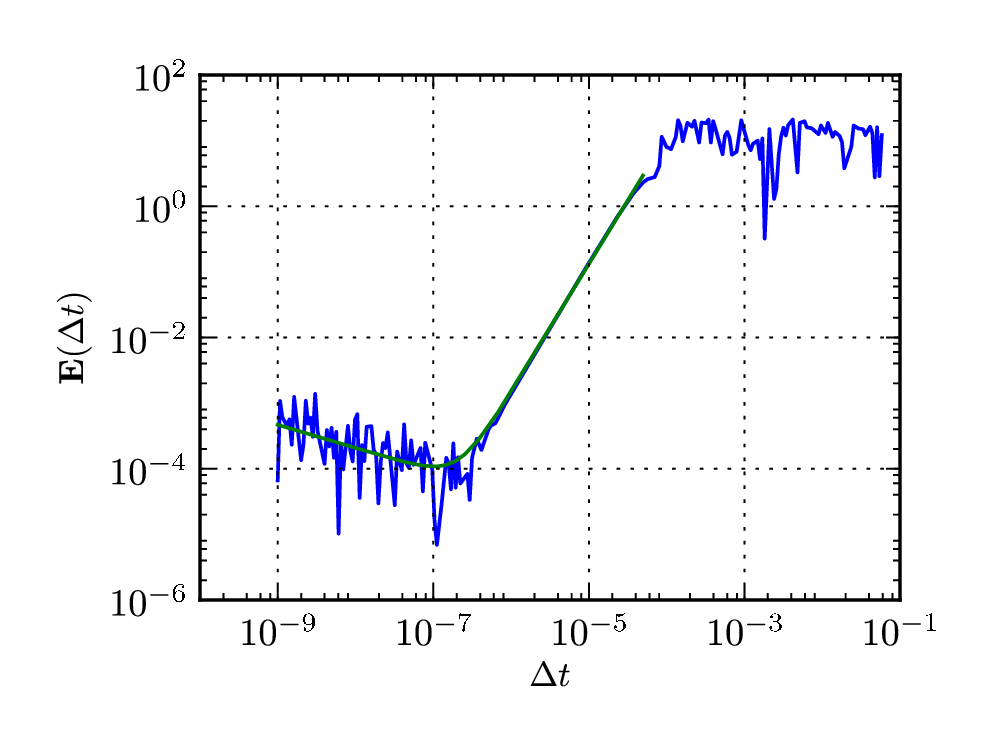}
    \includegraphics[width=\halffigurewidth]{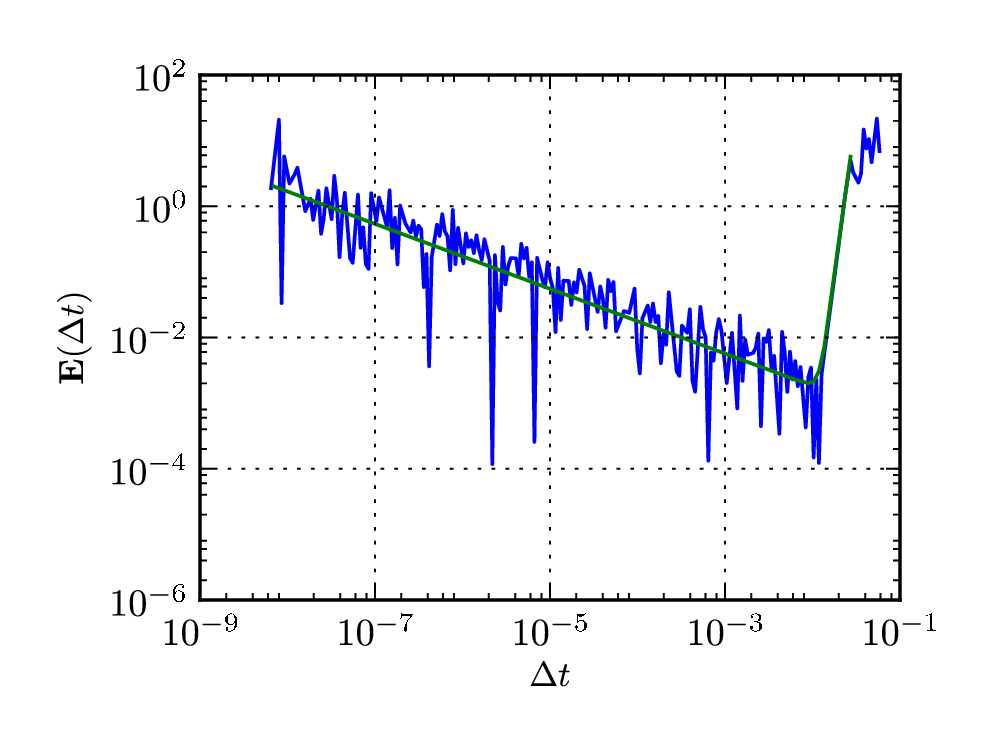}
    \caption{Error at time $T = 30$ for the \cg{1} solution (left) and
      at time $T = 40$ for the \cg{5} solution (right) of the Lorenz
      system.  The slopes of the green lines are $-0.35 \approx -1/2$
      and $1.95 \approx 2$ for the \cg{1} method. For the \cg{5}
      method, the slopes are $-0.49 \approx -1/2$ and $10.00 \approx
      10$.}
    \label{fig:errors}
  \end{center}
\end{figure}

In Figure~\ref{fig:lorenz-errors-3d}, results are presented for an
investigation of the influence on both the step size $\kk$ and the
polynomial degree $q$ in the \cg{q} method. As expected, the minimal
error is obtained when both the polynomial degree $q$ and the step
size are \emph{maximal}. Maximising the step size minimizes the
influence of numerical round-off errors (the term
$\kk^{-\sfrac{1}{2}}$), and as a consequence the polynomial degree $q$
must be large in order to suppress the discretisation error (the term
$\kk^{2q}$).

\begin{figure}
  \begin{center}
    \includegraphics[width=\figurewidth]{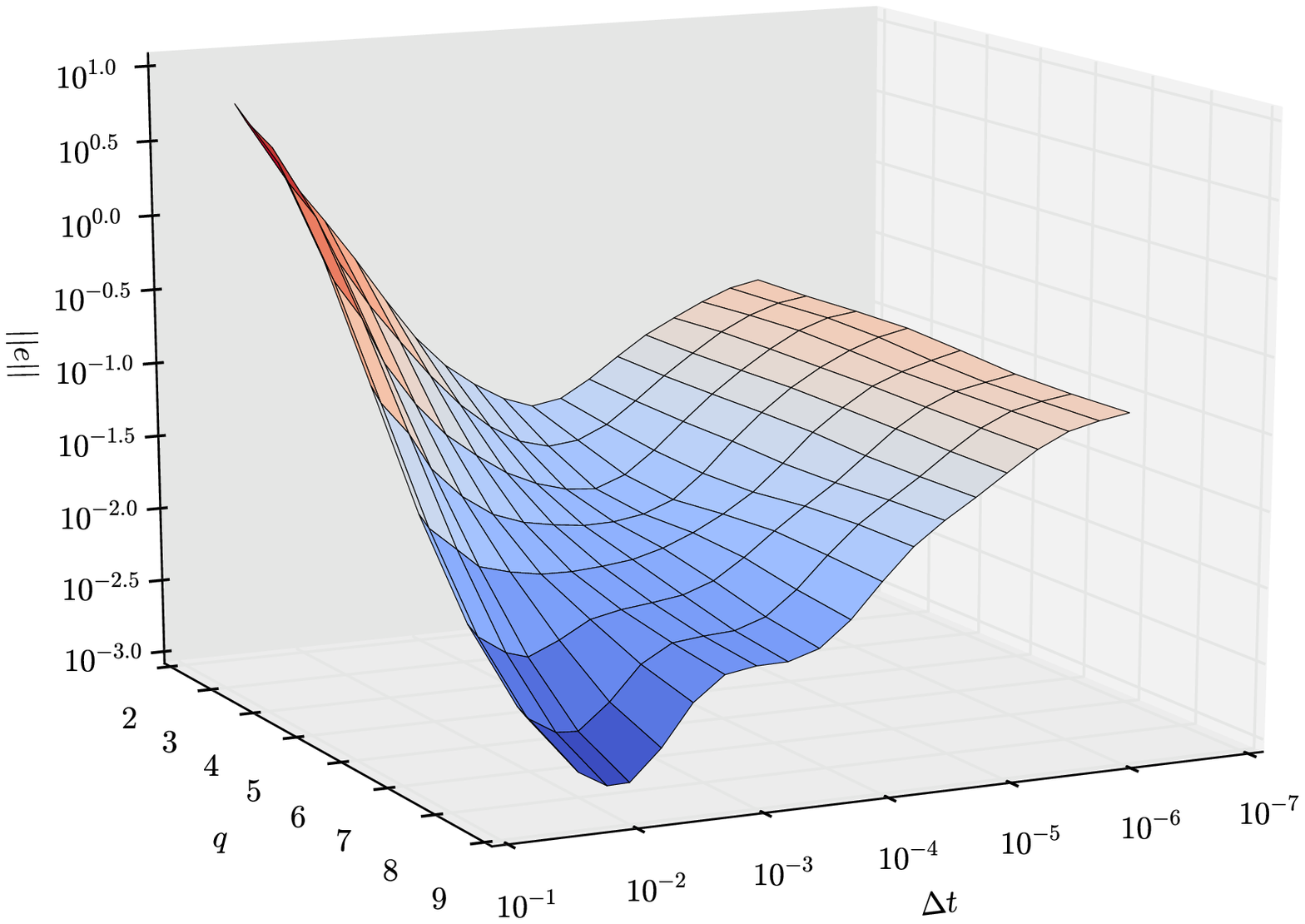}
    \includegraphics[width=\halffigurewidth]{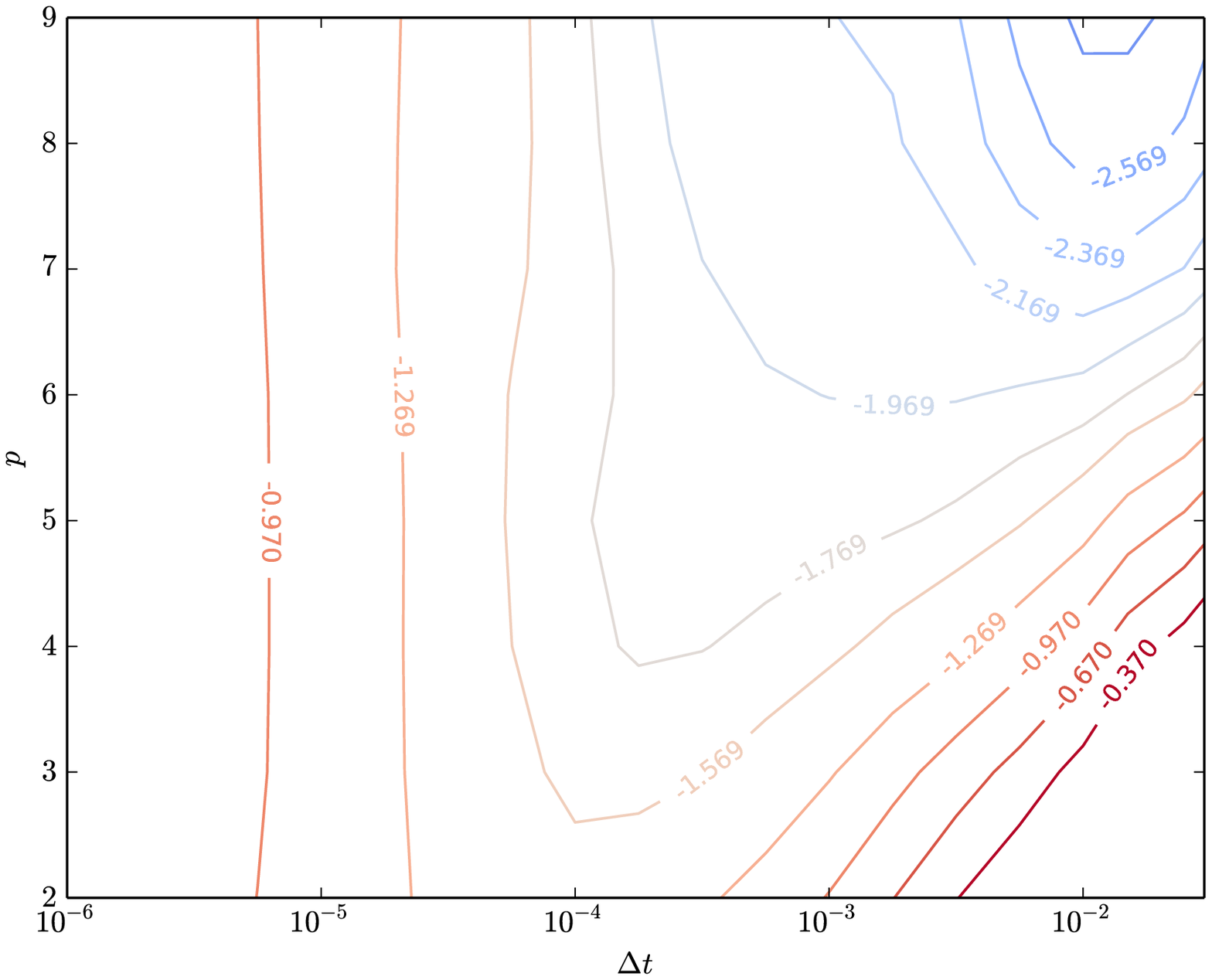}
    \includegraphics[width=\halffigurewidth]{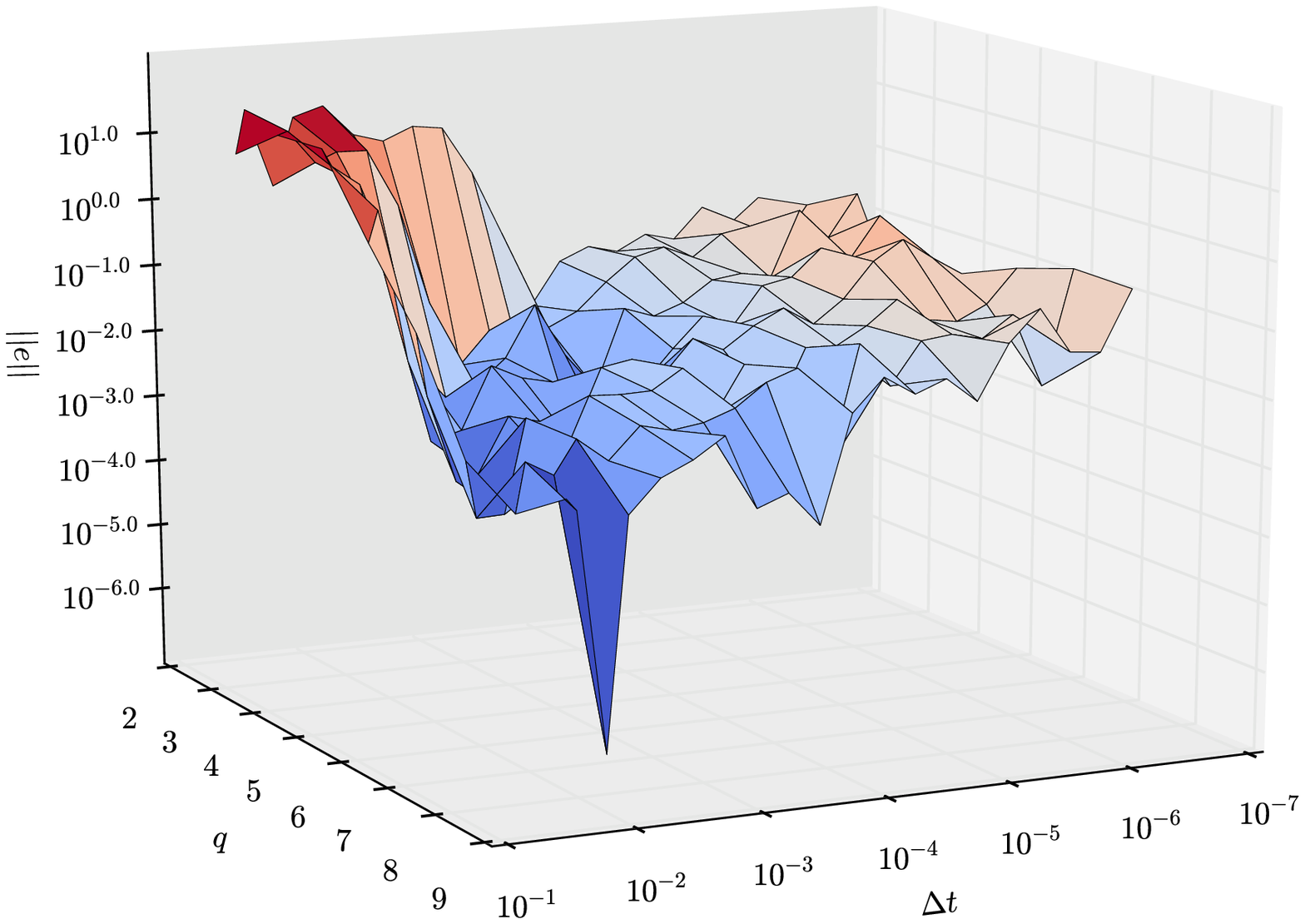}
    \caption{\emph{Top:} The accumulated total error at final time $T
      = 40$ for numerical solutions of the Lorenz system with
      different step size $\kk$ and polynomial degree $q$ using the
      \cg{q} method. Due to the random nature of the round-off errors,
      the data has been smoothed. \emph{Lower left:} Contour lines of
      the smoothed data. \emph{Lower right:} The raw data included for
      completeness.}
    \label{fig:lorenz-errors-3d}
  \end{center}
\end{figure}

\subsection{The Van der Pol oscillator}

We next consider the Van der Pol oscillator, given by
the second order ODE
\[
 \ddot{u} = \mu(1-u^2)\dot{u} - u.
\]
Rewritten as a system of first order equations, it reads
\begin{equation} \label{eq:vanderpol}
  \left\{
  \begin{aligned}
    \dot{u}_1 &= u_2, \\
    \dot{u}_2 &= \mu(1-u_1^2)u_2 - u_1.  \\
    \end{aligned}
  \right.
\end{equation}
We compute solutions on $[0, 2\mu]$ for $\mu = 10^3$ and $u(0) = (2,
0)$. This configuration is used as a test problem for ODE solvers in
\cite{Testset2008}. For large values of the parameter $\mu$, the
solution quickly approaches a limit cycle.


As for the Lorenz system, the stability factors(s) grow very rapidly
(exponentially), as indicated in
Figures~\ref{fig:vdpol-stability-full}
and~\ref{fig:vdpol-stability-detail}. However, the rapid growth is
localised in time close to $T \approx 807 \cdot n$ for $n \in
\mathbb{N}$. For times before or after these points of instability,
the stability factor is of moderate size. This means that solutions
are difficult to compute only at points near the points of
instability; that is, a solution may be easily computed at time $t =
1000$ but not at time $t = 807$.

\begin{figure}
  \begin{center}
    \includegraphics[width=\figurewidth]{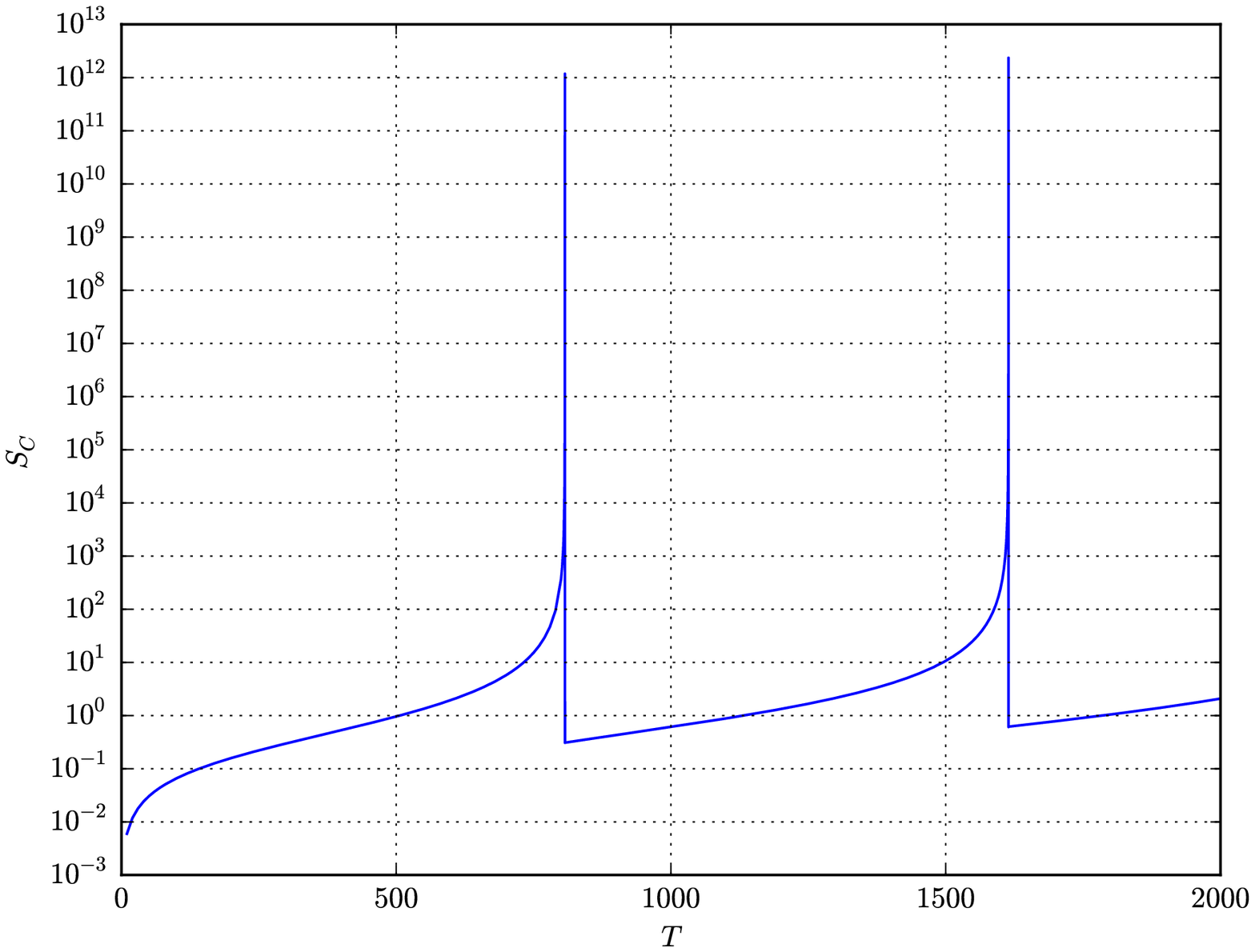}
    \caption{Growth of the computational stability factor $S_C(T)$ for
      the Van der Pol oscillator~\eqref{eq:vanderpol}.}
    \label{fig:vdpol-stability-full}
  \end{center}
\end{figure}

\begin{figure}
  \begin{center}
    \includegraphics[width=\halffigurewidth]{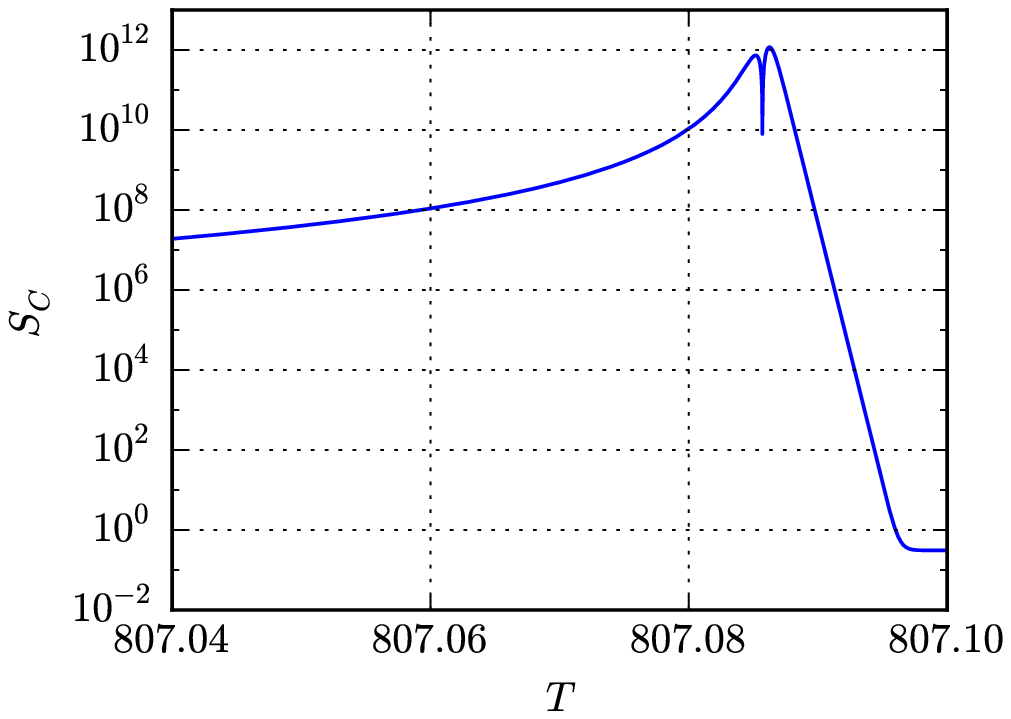}
    \includegraphics[width=\halffigurewidth]{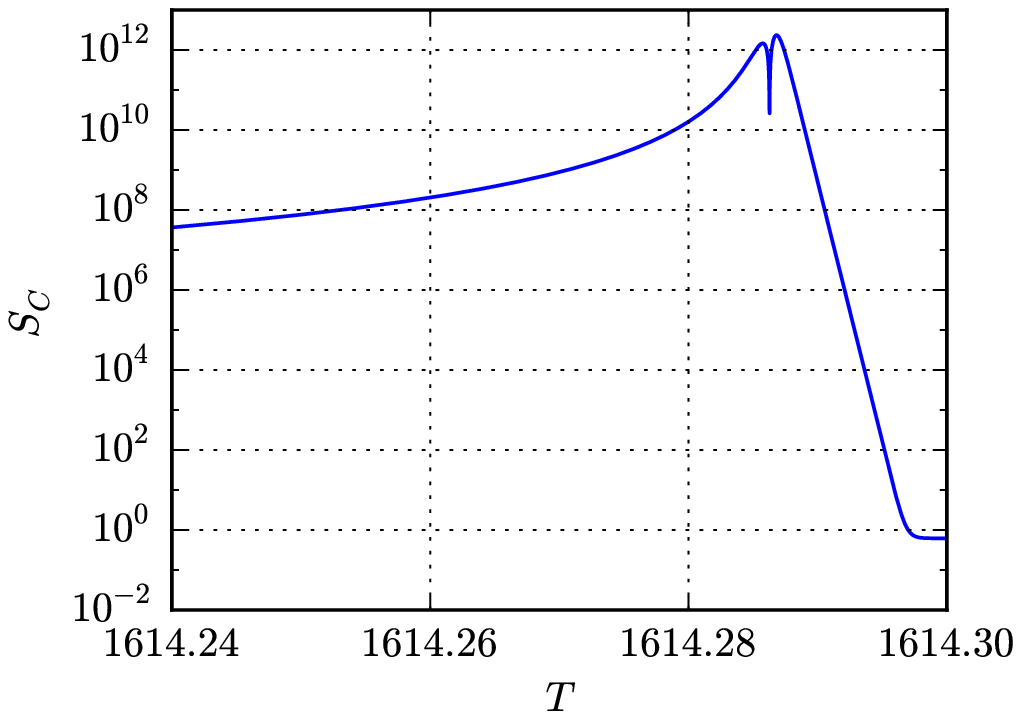}
    \caption{Detail of growth of the computational stability factor
      $S_C(T)$ for the Van der Pol oscillator~\eqref{eq:vanderpol}.}
    \label{fig:vdpol-stability-detail}
  \end{center}
\end{figure}

This rapid growth of stability factors is reflected in the growth of
the error for numerical solutions as shown in
Figures~\ref{fig:vdpol-error-full} and~\ref{fig:vdpol-error-detail}.
Examining these plots in more detail, we notice that in accordance
with the error estimate of Theorem~\ref{th:errorestimate} and
Equation~\eqref{eq:convergence}. For the numerical solutions studied
in Figures~\ref{fig:vdpol-error-full}
and~\ref{fig:vdpol-error-detail}, the discretisation error dominates
for low order methods. As the polynomial degree $q$ is increased, the
error decreases until the point when the computational error starts to
dominate.  We notice that the baseline error is of size $E \sim
10^{-14}$ for the highest order methods when the stability factor is
of size $S \sim 10^{2}$, and the error spikes at $E \sim 10^{-4}$ at
times when the stability factor takes on large values $S \sim
10^{12}$. This is in good agreement with the error estimate: $E \sim S
\cdot 10^{-16}$.

\begin{figure}
  \begin{center}
    \includegraphics[width=\figurewidth]{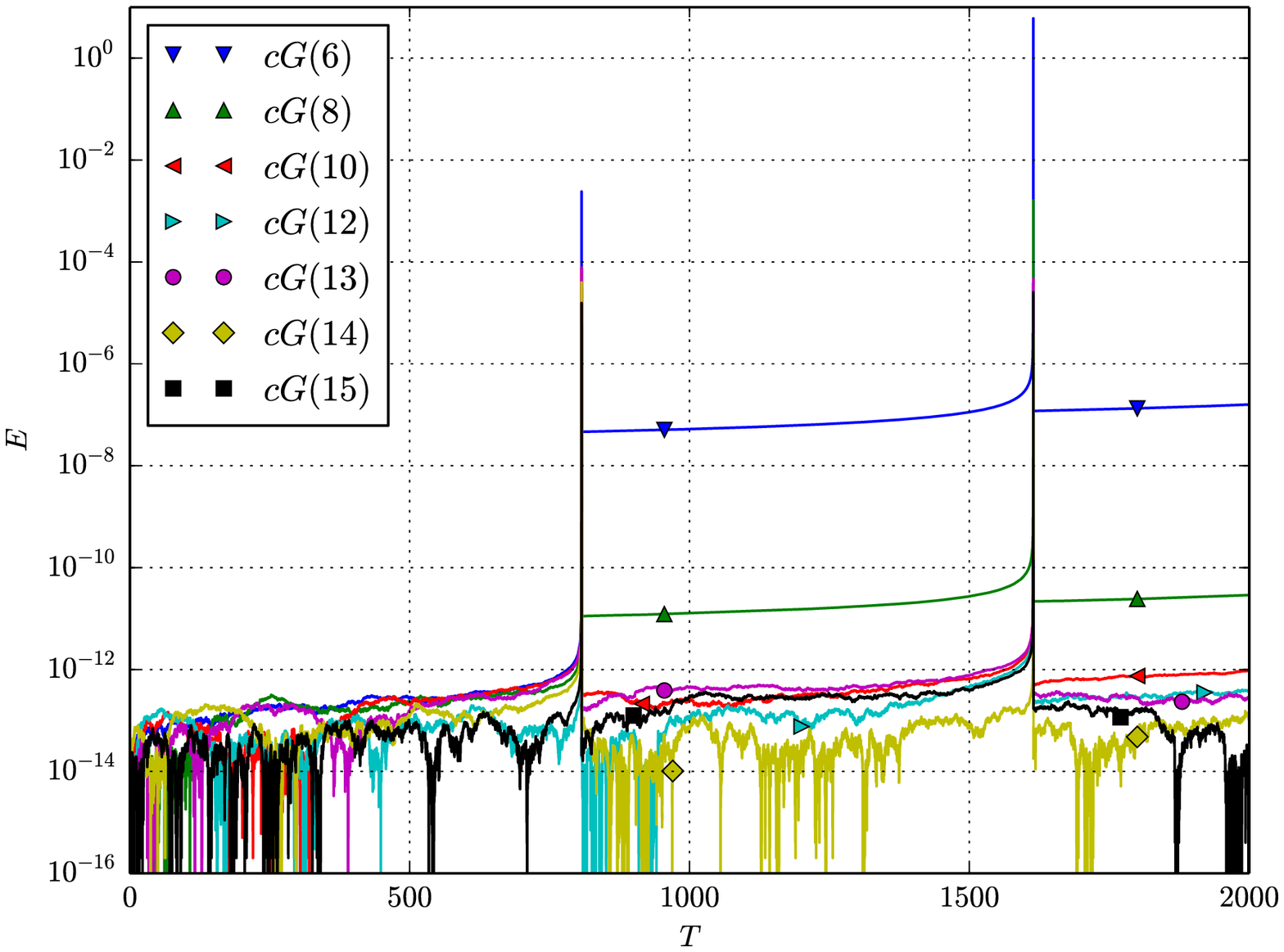}
    \caption{Growth of error for solutions of the Van der Pol
      oscillator~\eqref{eq:vanderpol} computed with time step $\kk=10^{-3}$.}
    \label{fig:vdpol-error-full}
  \end{center}
\end{figure}

\begin{figure}
  \begin{center}
    \includegraphics[width=\halffigurewidth]{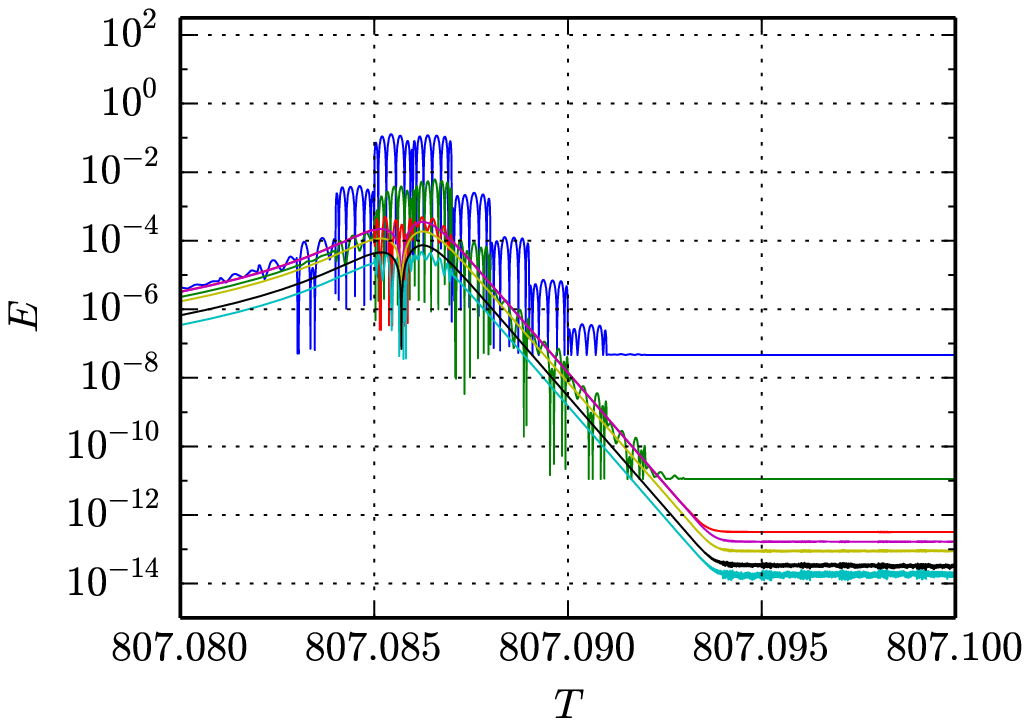}
    \includegraphics[width=\halffigurewidth]{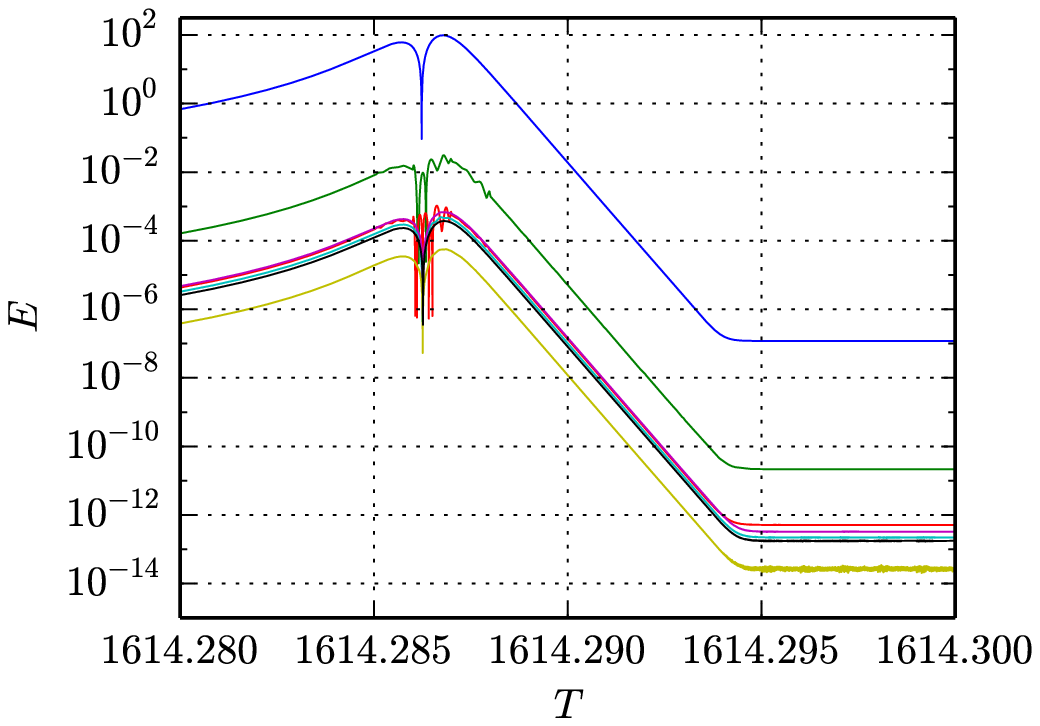}
    \caption{Detail of growth of error for solutions of the Van der Pol
      oscillator~\eqref{eq:vanderpol} computed with time step $\kk=10^{-3}$.}
    \label{fig:vdpol-error-detail}
  \end{center}
\end{figure}

\pagebreak

\section{Conclusions}

We have proved error estimates accounting for data, discretisation and
computational (round-off) errors in the numerical solution of initial
value problems for ordinary differential equations. These error
estimates quantify the accumulation rates for numerical round-off
error as inversely proportional to the square root of the step size,
and proportional to a specific computable stability factor. The effect
of round-off errors is mostly pronounced for large values of the
stability factor, which includes both chaotic dynamical systems as
well as long-time integration of systems which exhibit only a moderate
growth of the stability factor.

\pagebreak
\bibliographystyle{spmpsci}
\bibliography{paper}

\begin{thebibliography}{10}
\providecommand{\url}[1]{{#1}}
\providecommand{\urlprefix}{URL }
\expandafter\ifx\csname urlstyle\endcsname\relax
  \providecommand{\doi}[1]{DOI~\discretionary{}{}{}#1}\else
  \providecommand{\doi}{DOI~\discretionary{}{}{}\begingroup
  \urlstyle{rm}\Url}\fi

\bibitem{becker_optimal_2001}
Becker, R., Rannacher, R.: An optimal control approach to a posteriori error
  estimation in finite element methods.
\newblock Acta Numerica \textbf{10}, 1–--102 (2001)

\bibitem{cao2004posteriori}
Cao, Y., Petzold, L.: A posteriori error estimation and global error control
  for ordinary differential equations by the adjoint method.
\newblock SIAM Journal on Scientific Computing \textbf{26}(2), 359--374 (2004)

\bibitem{coomes_rigorous_1995}
Coomes, B.A., Kocak, H., Palmer, K.J.: Rigorous computational shadowing of
  orbits of ordinary differential equations.
\newblock Numerische Mathematik \textbf{69}(4), 401–--421 (1995)

\bibitem{delfour_discontinuous_1981}
Delfour, M., Hager, W., Trochu, F.: Discontinuous {G}alerkin methods for
  ordinary differential equations.
\newblock Math. Comp. \textbf{36}, 455–--473 (1981)

\bibitem{eriksson_introduction_1995}
Eriksson, K., Estep, D., Hansbo, P., Johnson, C.: Introduction to adaptive
  methods for differential equations.
\newblock Acta Numerica \textbf{4}, 105–--158 (1995)

\bibitem{estep_posteriori_1995}
Estep, D.: A posteriori error bounds and global error control for
  approximations of ordinary differential equations.
\newblock {SIAM} J. Numer. Anal. \textbf{32}, 1--48 (1995)

\bibitem{estep_global_1994}
Estep, D., French, D.: Global error control for the continuous {G}alerkin
  finite element method for ordinary differential equations.
\newblock M2AN \textbf{28}, 815–--852 (1994)

\bibitem{estep_pointwise_1998}
Estep, D., Johnson, C.: The pointwise computability of the {L}orenz system.
\newblock Math. Models. Meth. Appl. Sci. \textbf{8}, 1277--1305 (1998)

\bibitem{Granlund15}
Granlund, T., {the GMP development team}: {GNU MP}: {T}he {GNU} {M}ultiple
  {P}recision {A}rithmetic {L}ibrary (2015).
\newblock \url{http://gmplib.org/}

\bibitem{Hairer2008}
Hairer, E., McLachlan, R.I., Razakarivony, A.: Achieving brouwer’s law with
  implicit runge--kutta methods.
\newblock BIT Numerical Mathematics \textbf{48}(2), 231--243 (2008)

\bibitem{higham2002accuracy}
Higham, N.: {Accuracy and stability of numerical algorithms}, second edn.
\newblock Society for Industrial Mathematics (2002)

\bibitem{hulme_discrete_1972}
Hulme, B.L.: Discrete {G}alerkin and related one-step methods for ordinary
  differential equations.
\newblock Math. Comput. \textbf{26}(120), 881–--891 (1972)

\bibitem{hulme_one-step_1972}
Hulme, B.L.: One-step piecewise polynomial {G}alerkin methods for initial value
  problems.
\newblock Math. Comput. \textbf{26}(118), 415–--426 (1972)

\bibitem{johnson_error_1988}
Johnson, C.: Error estimates and adaptive time-step control for a class of
  one-step methods for stiff ordinary differential equations.
\newblock {SIAM} J. Numer. Anal. \textbf{25}(4), 908–--926 (1988)

\bibitem{jorba_software_2005}
Jorba, A., Zou, M.: A software package for the numerical integration of {ODEs}
  by means of high-order {T}aylor methods.
\newblock Experimental Mathematics \textbf{14}(1), 99–--117 (2005)

\bibitem{kehlet_analysis_2010}
Kehlet, B.: Analysis and implementation of high-precision finite element
  methods for ordinary differential equations with application to the {L}orenz
  system.
\newblock {MSc} thesis, Department of Informatics, University of Oslo (2010)

\bibitem{KehletLogg2014}
Kehlet, B., Logg, A.: Quantifying the computability of the {L}orenz system.
\newblock In: Adaptive Modeling and Simulation (2013)

\bibitem{kehlet_2015_16671}
Kehlet, B., Logg, A.: {Code package for the paper "A posteriori error analysis
  of round-off errors in the numerical solution of ordinary differential
  equations"} (2015).
\newblock \doi{10.5281/zenodo.16671}.
\newblock \urlprefix\url{http://dx.doi.org/10.5281/zenodo.16671}.
\newblock Available at \url{http://dx.doi.org/10.5281/zenodo.16671}

\bibitem{li2000}
Li, J., Zeng, Q., Chou, J.: Computational uncertainty principle in nonlinear
  ordinary differential equations i: Numerical results.
\newblock Science in China (E) \textbf{43}(5), 449--460 (2000)

\bibitem{li2001}
Li, J., Zeng, Q., Chou, J.: Computational uncertainty principle in nonlinear
  ordinary differential equations ii: Theoretical analysis.
\newblock Science in China (E) \textbf{44}(1), 55--74 (2001)

\bibitem{logg_multi-adaptive_2003-1}
Logg, A.: {Multi-Adaptive} {G}alerkin methods for {ODEs} {I}.
\newblock {SIAM} J. Sci. Comput. \textbf{24}(6), 1879–--1902 (2003)

\bibitem{logg_multi-adaptive_2003}
Logg, A.: {Multi-Adaptive} {G}alerkin methods for {ODEs} {II}: Implementation
  and applications.
\newblock {SIAM} J. Sci. Comput. \textbf{25}(4), 1119–--1141 (2003)

\bibitem{tanganyika-web}
Logg, A., Kehlet, B.: Tanganyika.
\newblock https://bitbucket.org/benjamik/tanganyika

\bibitem{lorenz_deterministic_1963}
Lorenz, E.N.: Deterministic nonperiodic flow.
\newblock J. Atmosph. Sci. \textbf{20}, 130–--141 (1963)

\bibitem{Testset2008}
Mazzia, F., Magherini, C.: Test set for initial value problem solvers, release
  2.4.
\newblock Technical Report~4, Department of Mathematics, University of Bari,
  Italy (2008).
\newblock Available at \url{http://pitagora.dm.uniba.it/~testset}

\bibitem{viswanath2004fractal}
Viswanath, D.: {The fractal property of the Lorenz attractor}.
\newblock Physica D: Nonlinear Phenomena \textbf{190}(1-2), 115--128 (2004)

\end{thebibliography}


\end{document}